\documentclass[11pt,reqno]{amsart}

\usepackage{amsmath,amscd,amssymb,amsfonts}
\usepackage{amsthm}
\usepackage{color,enumerate,accents,csquotes,placeins,mathtools}
\usepackage{mathrsfs}  

\usepackage{adjustbox}
\usepackage{array}
\usepackage{multicol}
\usepackage{multirow} 
\usepackage{float} 
\usepackage{extpfeil} 
\usepackage{graphicx} 
\usepackage{caption} 
\usepackage{setspace}
\usepackage{xcolor} 
\usepackage{enumitem}
\mathtoolsset{showonlyrefs} 
\usepackage[margin=2.7cm]{geometry}

\usepackage{tikz-cd}
\usetikzlibrary{shapes.misc}
\usetikzlibrary {patterns,patterns.meta} 
\usetikzlibrary{3d}
\pgfdeclarepatternformonly{mydots}{\pgfqpoint{-1pt}{-1pt}}{\pgfqpoint{1pt}{1pt}}{\pgfqpoint{3pt}{3pt}}
{%
  \pgfpathcircle{\pgfqpoint{0pt}{0pt}}{.2pt}%
  \pgfusepath{fill}%
}%
\tikzset{cross/.style={cross out, draw, 
         minimum size=4*(#1-\pgflinewidth), 
         inner sep=0pt, outer sep=0pt}}
\tikzset{circ/.style={shape=circle, inner sep=1pt, draw, node contents=}}

\newcommand{\RR}{\mathbb R}
\newcommand{\CC}{\mathbb C}
\newcommand{\GL}{\operatorname{GL}}
\newcommand{\defn}[1]{\emph{#1}}

\newcommand{\itres}{\operatorname{tres}}
\newcommand{\res}{\operatorname{res}}

\newtheorem{theorem}{Theorem}
\newtheorem*{theorem*}{Theorem}

\newtheorem{lemma}{Lemma}
\newtheorem{proposition}{Proposition}
\newtheorem*{proposition*}{Proposition}
\newtheorem{corollary}{Corollary}

\newtheorem*{prop*}{Proposition}

\theoremstyle{definition} 
\newtheorem{definition}{Definition}

\newtheorem*{example*}{Example}
\newtheorem{example}{Example}

\newtheorem*{warning*}{Warning}

\theoremstyle{remark} 

\newtheorem{remark}{Remark}

\usetikzlibrary{decorations.markings}
\tikzset{double line with arrow/.style args={#1,#2}{decorate,decoration={markings,%
mark=at position 0 with {\coordinate (ta-base-1) at (0,1pt);
\coordinate (ta-base-2) at (0,-1pt);},
mark=at position 1 with {\draw[#1] (ta-base-1) -- (0,1pt);
\draw[#2] (ta-base-2) -- (0,-1pt);
}}}}

\usepackage[colorlinks=true,hyperindex,pagebackref=false, citecolor=cyan,linkcolor=cyan]{hyperref}

\newcommand{\papertitle}{\Large{A residue formula for integrals\\[0.4em]with hyperplane singularities}}

\title[Integrals with hyperplane singularities]{\papertitle}

\author{Andrew O'Desky}
\address{
\parbox{0.7\linewidth}{
    Department of Mathematics\\
    Princeton University\\
}}
\email{andy.odesky@gmail.com}

\date{\today. Research supported by NSF grant DMS-2103361.}

\begin{document}



\maketitle


This article is concerned with evaluating integrals of the form 
\begin{equation}\label{eqn:presentationOfOmega} 
    \int_{\RR^r} \frac{h(z)\,dz}{g_1(z)\cdots g_R(z)}
\end{equation} 
where each $g_k$ is a {complex affine function} 
that does not vanish on $\RR^r$. 
When the dimension $r$ is one 
    the set $Z$ of poles of the meromorphic form 
    $\omega = h(z)(g_1(z)\cdots g_R(z))^{-1}\,dz$ 
    has a distinguished subset, 
    namely those poles in the {closed upper half-plane} 
    $Z_{\overline{\mathbb H}} =
    \{z \in Z : \mathrm{Im}\,z \geq 0\}$. 
Under suitable convergence conditions the integral is given by 
\begin{equation} 
    \int_\RR \frac{h(z)\,dz}{g_1(z)\cdots g_R(z)}
    =
    2\pi i
    \sum_{z\in Z_{\overline{\mathbb H}}}
    \mathrm{res}[\omega,z].
\end{equation} 
When the dimension $r$ is greater than one, 
    the set $Z$ of poles of $\omega$ 
    is indexed by \emph{affine flags} 
\begin{equation} 
\gamma \,\,\,: \,\,\,
H_{1} \supset 
H_{1} \cap H_{2} \supset \cdots \supset
    H_{1} \cap \cdots \cap H_{r} = \{z_\gamma\}
\end{equation} 
cut out by $r$ of the hyperplanes 
where $\omega$ is singular. 
Our residue formula takes the form 
\begin{equation} 
    \int_{\RR^r} \frac{h(z)\,dz}{g_1(z)\cdots g_R(z)}
    = 
    (2\pi i)^r
    \sum_{\gamma \in Z_\ast}
        \res[\omega,\gamma] 
\end{equation}   
where $Z_\ast \subset Z$ 
is a distinguished subset of flags 
    and $\res[\omega,\gamma]$ 
    is the iterated residue of $\omega$ along $\gamma$. 


The primary goal of this article 
    is to determine 
    a minimal subset $Z_\ast$ of flags 
    using sign conditions on the {minors} 
    of their defining matrices. 
We introduce two such conditions 
    for real matrices, which we call 
    \emph{stability} and \emph{compatability} 
    (see \S\ref{sec:Minors}). 
To apply them in the present context, 
    we use the elementary but crucial observation 
    that any affine hyperplane $H_k \subset \CC^r$ 
    disjoint from $\RR^r$ may be expressed as 
\begin{equation} 
    H_k = 
    \{v \in \CC^r : f_k(v) - is_k = 0\}
\end{equation} 
    for some real linear form $f_k$ 
    and constant $s_k$ with positive real part. 
For a collection of hyperplanes $H_1,\ldots,H_r$ so expressed 
    let $J_H$ be the Jacobian of 
    $(f_1,\ldots,f_r) \colon \RR^r \to \RR^r$. 

\begin{theorem*}
    Assume the iterated residue expansion of 
    $\int_{\RR^r}\omega$ along $\overline{\mathbb H}^r$ 
    converges. 
If every polar flag of $\omega$ is compatible 
    with $\overline{\mathbb H}^r$, then 
\begin{equation} 
    \frac{1}{(2\pi i)^{r}}
    \int_{\RR^r} \frac{h(z)\,dz}{g_1(z)\cdots g_R(z)}
    = 
    \sum_{\gamma \in Z_{\overline{\mathbb H}^r}}
        \res[\omega,\gamma] 
\end{equation}   
    where 
    $Z_{\overline{\mathbb H}^r} \subset Z$ 
    is the subset of flags cut out by 
    a collection $H$ of polar hyperplanes whose 
    Jacobian $J_H$ is stable. 
\end{theorem*}

For a typical integral, 
    there are $r!\binom{R}{r}$ flags in $Z$ 
    and at most $\binom{R}{r}$ of these will be stable. 
The essential step in our proof 
    is to show that the sum over 
    residues of the far more numerous 
    unstable flags vanishes identically, 
    even though individual residues do not vanish. 

The hypothesis of compatibility cannot be dropped. 
In fact, if the formula holds 
    for sufficiently many forms $\omega$ with 
    poles along a fixed set of hyperplanes then 
    the hyperplanes are compatible 
    (Proposition~\ref{prop:converse}). 
For a simple example which illustrates 
    the use of compatibility and stability, 
    we encourage the reader to look at Example~\ref{sec:Example}.

The only other residue formula applying in the present context 
    to our knowledge is \cite[Theorem~2]{zbMATH00739368}. 
The hypotheses of our residue formula are also 
    more practical to verify. 
Applying \cite[Theorem~2]{zbMATH00739368} requires 
    a grouping of the polar hyperplanes 
    $H_1,\ldots,H_R$ 
    into $r$ (generally reducible) divisors 
    $D_1,\ldots,D_r$ 
    satisfying a certain compatiblity condition 
    \cite[Definition~1]{zbMATH00739368}. 
When $R \gg r$ finding such a grouping 
    directly appears to be a difficult problem. 
This question is addressed 
    in \S\ref{sec:ClassicalResidue} 
    where we show that 
    stability leads to 
    a canonical compatible grouping 
    of the hyperplanes in the sense of 
    \cite[Definition~1]{zbMATH00739368}. 

Beyond their intrinsic interest, 
    integrals of the form \eqref{eqn:presentationOfOmega} 
    arise naturally in 
    the harmonic analysis of toric varieties 
    and are important for understanding 
    the distribution of rational points of bounded height. 
The approach in \cite{zbMATH01353487} 
    was to approximate these integrals. 
This suffices for determining 
    the main term in the asymptotic number 
    of rational points of bounded height, 
    but to understand the finer aspects of 
    this distribution it is important 
    to have an exact formula. 
For instance, in \cite{cubic} 
    an exact formula 
    for the height zeta function 
    of a particular toric surface 
    was used to compute the secondary term 
    in the number of monic abelian cubic trace-one polynomials 
    of bounded height. 
The residue formula developed here serves as a basis 
    for extending these methods 
    to general toric varieties. 


\section{Truncated iterated residues} 

An essential feature of our approach is our use of 
    a variant of the iterated residue. 
Let $V$ be an oriented real vector space of dimension $r$. 
First let us recall the classical residue. 

\subsection{The classical residue} 

Let $\omega$ be a meromorphic $r$-form on 
$V_\CC = V \otimes \CC$ 
and let $D$ be an irreducible divisor. 
The form $\omega$ has a simple singularity along $D$ if 
    in some neighborhood $U \subset V_\CC$ 
    of every point on $D$ the form $\omega$ 
    can be expressed as 
\begin{equation} 
    \omega = \frac{dg}{g} \wedge \omega' +\eta
\end{equation} 
where $g$ is 
an irreducible holomorphic function on $U$ 
such that $D \cap U = \{g=0\}$, 
$\eta$ is holomorphic on $U$, 
$\omega'$ is meromorphic on $U$, 
and $\omega'|_{D \cap U}$ is meromorphic on $D \cap U$. 
If $\omega$ has a simple singularity along $D$ 
    then the \defn{residue of $\omega$ along $D$} 
    is the meromorphic $(r-1)$-form on $D$ 
    locally given by $\omega'|_{D \cap U}$, 
    and is denoted by $$\res[\omega,D];$$ 
if $\omega$ does not have a simple singularity along $D$, 
then we set $\res[\omega,D]=0$. 
The residue only depends on $\omega$ and $D$, 
and not on the choices of local defining equations for $D$. 

\subsection{Flags} 

Next we take iterated residues along flags. 
Consider an affine flag in $V_\CC$ of the form 
\begin{equation} 
\gamma \,\,\,: \,\,\,
H_{1} \supset 
H_{1} \cap H_{2} \supset \cdots \supset
H_{1} \cap \cdots \cap H_{k}
\end{equation} 
for affine hyperplanes $H_1,\ldots,H_k \subset V_\CC$. 
It will always be assumed that 
    the intersection is transverse 
    so each affine subspace 
    $\gamma(j) \coloneqq H_{1} \cap \cdots \cap H_{j}$ 
    has dimension $r-j$. 
All flags in this paper will be descending, 
    which in terms of the associated flag variety 
    amounts to a transpose 
    of the usual convention. 
We recall the description. 
Let 
$$\gamma_0 \,: \,
L_{1} \supset 
L_{1} \cap L_{2} \supset \cdots \supset
L_{1} \cap \cdots \cap L_{k}$$ 
    be the linear flag in $V_\CC$ 
    cut out by the linear hyperplanes 
    $L_1,\ldots,L_k$ 
    obtained by translating $H_1,\ldots,H_k$. 
If $V_\CC = \CC^r$ 
then $\gamma_0$ may be represented 
    by a $k \times r$ matrix $J$ 
    with rows $f_1,\ldots,f_k$ where 
    $L_j = \ker f_j$ when $f_j$ is regarded as a linear form 
    $f_j \colon \CC^r \to \CC$. 
Two matrices $J,J'$ represent 
    the same flag $\gamma_0$ if and only if 
    $J' = \ell J$ for some lower-triangular matrix 
    $\ell \in \GL_k(\CC)$. 
The $r$-step (descending) 
linear flags $\gamma_0 \,: \,
L_{1} \supset 
L_{1} \cap L_{2} \supset \cdots \supset
L_{1} \cap \cdots \cap L_{r}$ in $\CC^r$ 
are canonically identified with points of 
the flag variety 
$$
\mathscr F
= B^T \,\backslash \GL_r(\CC),
$$
\begin{align} 
\gamma_0
    \leftrightarrow
B^T
\begin{bmatrix}
\,\,\rule[.5ex]{1.75em}{0.4pt}\,\, f_1 \,\,\rule[.5ex]{1.75em}{0.4pt}\,\,\\
\,\,\rule[.5ex]{1.75em}{0.4pt}\,\, f_2 \,\,\rule[.5ex]{1.75em}{0.4pt}\,\,\\
\vdots \\
\,\,\rule[.5ex]{1.75em}{0.4pt}\,\, f_r \,\,\rule[.5ex]{1.75em}{0.4pt}\,\,\\
\end{bmatrix} ,\qquad
    \gamma_0(j) = \ker f_1  \cap \cdots \cap \ker f_j
\end{align} 
where $B \subset \GL_r(\CC)$ 
    is the subgroup of upper-triangular matrices. 
For a $k\times r$ matrix $J = (a_{ij})$ 
    and $j \in \{1,\ldots,k\}$ 
    let $p_j(J)$ be the $j$th leading principal minor 
\begin{equation} 
    p_j(J)= 
   \det
    \begin{bmatrix}
        a_{11}&\cdots&a_{1j}\\ 
        \vdots& \ddots&\vdots\\
        a_{jj}&\cdots&a_{jj}\\ 
    \end{bmatrix}.
\end{equation} 
The condition $p_j(J) \neq 0$ 
    does not depend on the matrix $J$ used to represent 
    $\gamma_0$, 
    so the condition $p_j(\gamma_0) \neq 0$ makes sense 
    for any linear flag $\gamma_0 \in \mathscr F$. 

\subsection{Truncated residues} 

\begin{definition}
The \defn{(iterated) residue of $\omega$ along $\gamma$} is 
$$
\res[\omega,\gamma]\coloneqq
\res[\res[\cdots\res[\res[\omega,\gamma(1)],\gamma(2)]\cdots 
, \gamma(k-1)],\gamma(k)],$$
a meromorphic top-degree form on $\gamma(k)$. 
Let $z \colon V_\CC \xrightarrow{\sim} \CC^r$ 
    be a $\CC$-linear isomorphism. 
The \defn{truncated (iterated) residue of $\omega$ along $\gamma$ with respect to $z$} is 
\begin{equation} 
    \itres_z[\omega,\gamma]=
    \begin{cases}
    \res[\omega,\gamma]
        &\text{if $p_1(z\gamma_0),\ldots,p_k(z\gamma_0) \neq 0$},\\
        0 &\text{otherwise.}
    \end{cases}
\end{equation} 
\end{definition}

\begin{remark}
    It is well-known that $p_1(J),\ldots,p_r(J) \neq 0$ 
    if and only if $J$ is contained in 
    the open set $B^T B$ 
    of invertible matrices admitting an $LU$-decomposition. 
In the primary case of interest when $k=r$, we thus have 
\begin{equation}\label{eqn:itresFormula1} 
    \itres_z[\omega,\gamma]=
    1_{B^T B}(z\gamma_0)
\res[\res[\cdots\res[\res[\omega,\gamma(1)],\gamma(2)]\cdots 
, \gamma(r-1)],\gamma(r)]
\end{equation} 
        where $1_{B^T B}$ is the characteristic function 
        of the open Bruhat cell 
        $B^T \,\backslash \,B^T B \subset \mathscr F$. 
    \end{remark}


\subsection{Solubility} 

We give some motivation for the truncated residue. 
Let $z_1,\ldots,z_r$ be the standard coordinates on $\CC^r$. 
Let $\omega=G(z) \, dz_1 \wedge \cdots \wedge dz_r$ 
be a meromorphic top-degree form on $\CC^r$. 
Let 
    $\gamma \,: \,
    H_{1} \supset 
    H_{1} \cap H_{2} \supset \cdots \supset
    H_{1} \cap \cdots \cap H_{r}$ 
    be a flag 
    cut out by affine hyperplanes $H_1,\ldots,H_r\subset \CC^r$ 
    intersecting transversely. 
We have defined the iterated residue $\res[\omega,\gamma]$. 
Now we consider computing 
    the iterated residue of $\omega(z)$ along $\gamma$ 
    \emph{in coordinates}. 
A careful analysis of this computation 
    reveals something subtle: 
    the iterated residue in coordinates 
    does not equal $\res[\omega,\gamma]$ 
    in general 
    but rather equals the truncated residue 
    $\itres_z[\omega,\gamma]$. 
Briefly, truncation occurs when there is insolubility 
    of singularities, and insolubility occurs 
    when a leading principal minor vanishes. 

Let $(v_1,\ldots,v_r)$ be the standard basis of $\CC^r$. 
If for all $x_2\ldots,x_r \in \RR$ 
there exists $z_1^\ast \in \CC$ 
such that $$z_1^\ast v_1+x_2v_2+\cdots+x_rv_r \in H_1$$ 
then we define 
\begin{equation} 
    \res_{z}[\omega,H_1]
    \coloneqq
    \res[G(z_1,x_2,\ldots,x_r)\,dz_1,z_1=z_1^\ast]
    \,dz_2 \wedge \cdots \wedge dz_r
\end{equation} 
where $\res[G\,dz_1,z_1=z_1^\ast]$ 
denotes the usual one-variable residue of 
$G\,dz_1$ at $z_1^\ast$; 
otherwise if $z_1^\ast$ is not always soluble, 
we set $\res_{z}[\omega,H_1]$ equal to zero. 
It is easy to see that 
\begin{equation} 
    \res_z[\omega,H_1]=
    \begin{cases}
        \res[\omega,H_1]
        &\text{if $z_1^\ast$ is soluble},\\
        0 &\text{otherwise.}
    \end{cases}
\end{equation} 
Note that vanishing here in the insoluble case 
    occurs for different reasons than vanishing 
    of $\res[\omega,H_1]$ when 
    $\omega$ does not have a simple singularity along $H_1$. 

Now we iterate this definition. 
The real affine function 
    $(x_2,\ldots,x_r) \mapsto z_1^\ast(x_2,\ldots,x_r)$ 
    extends to an affine function 
    $z_1^\ast \colon \CC^{r-1} \to \CC$, 
and we make the identification 
\begin{align} 
    \CC^{r-1} &\to H_1 \\
    (x_2,\ldots,x_r) &\mapsto 
    z_1^\ast(x_2,\ldots,x_r) v_1 + x_2 v_2 + \cdots + x_r v_r.
\end{align} 
Next we replace $\omega,H_1,z$ with 
    $\res_{z}[\omega,H_1],H_1 \cap H_2, z_{(2)} = (z_2,\ldots,z_r)$. 
If we write 
    $\res_z[\omega,H_1]=G_{H_1}(z_{(2)}) \,dz_2 \wedge \cdots \wedge dz_r$ 
then we have that 
\begin{equation} 
    \res_{z_{(2)}}[
    \res_z[\omega,H_1]
    , H_1 \cap H_2]=
\res[G_{H_1}(z_2,x_3,\ldots,x_r)\,dz_2,z_2=z_2^\ast]
    \,dz_3 \wedge \cdots \wedge dz_r
\end{equation} 
if $z_2^\ast$ is soluble, and otherwise it is zero. 
The condition of solubility now means that 
for any $x_3,\ldots,x_r \in \RR$ there exists $z_2^\ast \in \CC$ 
such that 
$$z_1^\ast(z_2^\ast,x_3,\ldots,x_r) v_1+z_2^\ast v_2+\cdots+x_rv_r \in H_1 \cap H_2.$$ 
It satisfies 
\begin{equation} 
    \res_{z_{(2)}}[
    \res_z[\omega,H_1]
    , H_1 \cap H_2]=
    \begin{cases}
    \res[\omega,H_1\supset H_1 \cap H_2]
        &\text{if $z_1^\ast,z_2^\ast$ are soluble},\\
        0 &\text{otherwise.}
    \end{cases}
\end{equation} 
Iterating this process leads to another 
    way of taking the iterated residue of $\omega$ 
    along $\gamma$, and 
    with the help of the proposition 
    we will see that this 
    agrees with the truncated residue. 



\begin{proposition}\label{prop:openBruhat}
Let $\gamma$ be a complete affine flag in $V_\CC$ 
    and let $z\colon V_\CC \xrightarrow{\sim} \CC^r$ 
    be a $\CC$-linear isomorphism. 
The values $z_1^\ast,\ldots,z_r^\ast$ are soluble 
    if and only if 
    the linearized flag $z\gamma_0$ 
    is in the open Bruhat cell 
    $B^T \, \backslash \,B^T B \subset \mathscr F$. 
\end{proposition}

\begin{proof} 
It suffices to consider the case $V = \RR^r$ and 
    $z$ is the identity. 
Let $\gamma :
H_{1} \supset 
H_{1} \cap H_{2} \supset \cdots \supset
H_{1} \cap \cdots \cap H_{r}$ 
where $H_k = \{g_k=0\}$ 
    and let $\partial g/\partial z$ 
    be the Jacobian matrix of 
    $g=(g_1,\ldots,g_r) \colon \CC^r \to \CC^r$. 
Note that $\partial g/\partial z$ 
    represents the flag $\gamma_0$. 
Fix $x_2,\ldots,x_r \in \RR$. 
At the $k$th step 
    of the truncated residue $\itres_z[\omega,\gamma]$, 
    the residue is taken at the unique value 
    $z_k=z_k^\ast=z_k^\ast(x_{k+1},\ldots,x_r) \in \CC$ 
    when it exists for which 
\begin{equation}\label{eqn:fk} 
    z_1^\ast v_1+\cdots+z_k^\ast v_k
    +x_{k+1}v_{k+1}+\cdots+x_rv_r \in 
    \gamma(k)=H_1 \cap \cdots \cap H_k
\end{equation} 
    where $z_1^\ast,\ldots,z_{k-1}^\ast$ 
    are determined by earlier residues. 
We claim that $z_1^\ast,\ldots,z_r^\ast$ 
    are soluble if and only if 
    $\partial g/\partial z \in B^T B$. 
Let $\gamma(k)_0$ denote the linear subspace 
    obtained from translating $\gamma(k)$. 
The value for $z_k^\ast$ exists in $\CC$ 
    for all $x_{k+1},\ldots,x_r\in \RR$ if and only if 
the restrictions of the linear forms 
    $z_{k+1},\ldots,z_{r}$ to $\gamma(k)_0$ 
are linearly independent in $\gamma(k)_0^\vee$. 
Indeed then 
$(z_{k+1},\ldots,z_r) \colon \gamma(k) \to \CC^{r-k}$ 
is an isomorphism of affine spaces, so its inverse map 
\begin{equation} 
(z_{k+1},\ldots,z_r) \mapsto 
z_1^\ast v_1+\cdots+z_k^\ast v_k
    +z_{k+1}v_{k+1}+\cdots+z_rv_r
\end{equation} 
determines unique quantities $z_1^\ast,\ldots,z_k^\ast$ 
for any $x_{k+1},\ldots,x_r\in \CC$ such that 
\eqref{eqn:fk} is satisfied, 
    and this furnishes the required value of 
    $z_k=z_k^\ast \in \CC$. 

Let $f_k = g_k - g_k(0)$. 
Let $(w_k)_k \subset \CC^r$ 
    be the dual basis to $(f_k)_k\subset \CC^r$. 
We conclude the values $z_1^\ast,\ldots,z_r^\ast$ are soluble 
    if and only if 
\begin{equation} 
    \begin{bmatrix}
        z_{k+1}(w_{k+1})&z_{k+1}(w_{k+2})&\cdots&z_{k+1}(w_{r})\\
        z_{k+2}(w_{k+1})&z_{k+2}(w_{k+2})&\cdots&z_{k+2}(w_{r})\\
        \vdots&\vdots&\ddots&\vdots\\
        z_{r}(w_{k+1})&z_{r}(w_{k+2})&\cdots&z_r(w_{r})\\
    \end{bmatrix}
\end{equation} 
is invertible for every $k = 1,\ldots,r$; 
equivalently, the matrix 
\begin{equation} 
    w_0
    \begin{bmatrix}
        z_{1}(w_{1})&z_{1}(w_{2})&\cdots&z_1(w_{r})\\
        z_{2}(w_{1})&z_{2}(w_{2})&\cdots&z_2(w_{r})\\
        \vdots&\vdots&\ddots&\vdots\\
        z_{r}(w_{1})&z_{r}(w_{2})&\cdots&z_{r}(w_{r})\\
    \end{bmatrix}
    w_0
\end{equation} 
has nonvanishing leading principal minors, 
where $w_0$ is the longest element of 
the symmetric group $S_r$. 
The subset of $\mathrm{GL}_r(\CC)$ 
with nonvanishing leading principal minors is $B^T B$, 
so the values $z_1^\ast,\ldots,z_r^\ast$ are soluble 
if and only if 
\begin{equation} 
    \frac{\partial f}{\partial z}=
    \frac{\partial g}{\partial z}=
    \begin{bmatrix}
        f_{1}(v_{1})&f_{1}(v_{2})&\cdots&f_1(v_{r})\\
        f_{2}(v_{1})&f_{2}(v_{2})&\cdots&f_2(v_{r})\\
        \vdots&\vdots&\ddots&\vdots\\
        f_{r}(v_{1})&f_{r}(v_{2})&\cdots&f_{r}(v_{r})\\
    \end{bmatrix}=
    \begin{bmatrix}
        z_{1}(w_{1})&z_{1}(w_{2})&\cdots&z_1(w_{r})\\
        z_{2}(w_{1})&z_{2}(w_{2})&\cdots&z_2(w_{r})\\
        \vdots&\vdots&\ddots&\vdots\\
        z_{r}(w_{1})&z_{r}(w_{2})&\cdots&z_{r}(w_{r})\\
    \end{bmatrix}^{-1}
\end{equation} 
    is in 
    $(w_0 B^T B w_0)^{-1}=B^TB$. 
\end{proof} 

\begin{corollary}
Let $\omega$ be a meromorphic top-degree form on $\CC^r$ 
and let $\gamma$ be a complete affine flag in $\CC^r$. 
The truncated residue of $\omega$ along $\gamma$ 
with respect to $z$ satisfies 
\begin{equation} 
\itres_z[\omega,\gamma]
=
\res_{z_{(r)}}[
\res_{z_{(r-1)}}[
    \cdots\res_{z_{(2)}}[\res_{z}[\omega,\gamma(1)]\,dz_2,\gamma(2)]\cdots 
, \gamma(r-1)], \gamma(r)].
\end{equation} 
\end{corollary}


\section{Minors} 

In this section we define 
    the conditions of stability and compatibility. 
We specialize to a meromorphic $r$-form on $V_\CC$ 
    with hyperplane singularities: 
\begin{equation}\label{eqn:omegaForm} 
    \omega= \frac{h(z)\,dz}{g_1(z)\cdots g_R(z)}
\end{equation} 
for affine functions $g_1,\ldots,g_R\colon V_\CC \to \CC$ 
    which are nonvanishing on $V$ 
    and a function $h$ which is holomorphic on a neighborhood 
    of $V$ in $V_\CC$. 
We let $Z$ denote the set of all $r$-step complete flags 
    formed from the polar hyperplanes of $\omega$. 
Flags in $Z$ will be called \defn{poles} of $\omega$. 

\begin{lemma}\label{lemma:elementary}
    Any affine hyperplane $H \subset V_\CC$ 
    disjoint from $V$ may be expressed as 
    $$H = \{v \in V_\CC : f(v) = is\}$$ 
    for some real linear form $f\colon V \to \RR$ 
    and constant $s$ satisfying $\mathrm{Re}\,s> 0$. 
If $H=\{f'(v) = is'\}$ for another real linear form $f'$ 
and constant $s'$ with positive real part, 
    then $f'=cf$ and $s'=cs$ for some positive real number $c$. 
\end{lemma}

By the lemma (which is elementary) we express 
each polar hyperplane as $H_k = \{g_k = 0\}$ 
    for $g_k = f_k - is_k$ 
    with such $f_k,s_k$. 
This presentation will be needed since 
    stability and compatibility 
    are only defined for real matrices. 


\subsubsection{Polyhedra} 

Let $z \colon V \xrightarrow{\sim} \RR^r$ be 
an $\RR$-linear isomorphism. 
Let $\Pi \subset V_\CC$ 
    be the subset which corresponds to 
    the product of the closed upper half-planes in $\CC^r$ 
    under the induced $\CC$-linear isomorphism 
    $z_\CC \colon V_\CC \xrightarrow{\sim} \CC^r$. 
Such a set $\Pi$ will be called 
    a \defn{polyhedron (with boundary $V$)}. 
Any polyhedron $\Pi$ may be uniquely expressed as 
    $$\Pi = V + i \Theta$$ 
    where $\Theta$ is the {cone} 
    $\Theta = \RR_{\geq 0} v_1 + \cdots + \RR_{\geq 0} v_r$ 
    for the dual basis $(v_1,\ldots,v_r)\subset V$ to $z$. 


\subsection{Stability and compatibility}\label{sec:Minors} 

For an $r\times r$ real matrix $J = (a_{ij})$, 
consider the $k$th leading principal minor 
\begin{equation} 
   p_k= 
   \det
    \begin{bmatrix}
        a_{11}&\cdots&a_{1k}\\ 
        \vdots& \ddots&\vdots\\
        a_{k1}&\cdots&a_{kk}\\ 
    \end{bmatrix}
    \qquad(k \in \{1,\ldots,r\}), 
\end{equation} 
the $k\times k$ minor 
\begin{equation} 
    q_{k\ell}=\det 
    \begin{bmatrix}
        a_{11}&\cdots&a_{1,k-1}&a_{1\ell}\\ 
        \vdots& \ddots&\vdots&\vdots\\
        a_{k1}&\cdots&a_{k,k-1}&a_{k\ell}\\ 
    \end{bmatrix}
    \qquad(\ell \in\{ k+1,\ldots,r\}), 
\end{equation} 
and the $(k-1)\times(k-1)$ minor 
\begin{equation} 
    r_{jk}= \det 
    \begin{bmatrix}
        a_{11}&\cdots&a_{1,k-1}\\
        \vdots&\ddots&\vdots\\
        a_{j-1,1}&\cdots&a_{j-1,k-1}\\
        a_{j+1,1}&\cdots&a_{j+1,k-1}\\
        \vdots&\ddots&\vdots\\
        a_{k1}&\cdots&a_{k,k-1}\\ 
    \end{bmatrix}
    \qquad
    (j \in \{1,\ldots,k-1\}). 
\end{equation} 


Now we formulate the key definitions of the paper. 

\begin{definition}
Let $J$ be a $k \times r$ real matrix. 
We say $J$ is \defn{stable} if 
    $$\text{$p_1,\ldots,p_k>0$ 
    and $(-1)^{\ell-j}r_{j\ell} \geq 0$ 
    for all $1\leq j < \ell \leq k$.}$$ 
We say $J$ is \defn{compatible} if 
    either it is not stable or 
\begin{equation} 
    q_{j\ell} \leq 0
        \,\,\,\text{for all $1 \leq j < \ell \leq k$.}
\end{equation} 
Let $H=(H_1,\ldots,H_k)$ be a collection of affine hyperplanes 
    defined by real linear forms $(f_1,\ldots,f_k)$ 
    and constants $(s_1,\ldots,s_k)$ with positive real parts. 
Let $\Pi=z^{-1}_\CC(\overline{\mathbb H}^r)$ be a polyhedron. 
Let $\partial f/\partial z$ be the Jacobian matrix of 
    $f=(f_1,\ldots,f_k) \colon V \to \RR^k$ 
    with respect to $z$-coordinates. 
We say $H$ is \defn{$\Pi$-stable} 
    (resp.~\defn{$\Pi$-compatible}) 
    if $\partial f/\partial z$ is 
    stable (resp.~compatible). 
\end{definition}



\begin{remark} 
The signs of $p_k(J)$ and $q_{k\ell}(J)$ 
    are invariant under $J \mapsto \ell J$ 
    for any $$\ell \in B_+^T = 
    \{\text{lower-triangular with positive diagonal}\}
    \subset \GL_r(\RR)
    ,$$ 
    while the signs of the $r_{jk}$ minors are invariant under 
    the subgroup of $B_+^T$ stabilizing the $j$th coordinate axis. 
The signs of these minors are therefore 
    invariant under the diagonal subgroup of $B_+$. 
Thus the conditions of 
    $\Pi$-stability and $\Pi$-compatibility 
    do not depend on the real forms $f$ 
    used to express each hyperplane as 
    $H=\{f(v) = is\}$ with $\mathrm{Re}\,s>0$. 
Note the condition of $\Pi$-compatibility makes sense for flags 
    however $\Pi$-stability does not. 
\end{remark} 



\begin{remark}
It appears likely that if $J\in\GL_r(\RR)$ 
is stable and $w\in S_r$ is a nontrivial permutation, 
    then $wJ$ is unstable. 
When this holds for every Jacobian of a given integral, 
    the residue formula simplifies since 
    each unordered set 
    $\{H_1,\ldots,H_r\}$ of hyperplanes 
    which cuts out a flag 
    for \emph{some} ordering 
    in fact has a \emph{canonical} ordering 
    $H$ which is $\Pi$-stable. 
\end{remark}


\subsection{The iterated residue expansion} 

Here we formulate a convergence condition 
    which will be clearly seen as necessary 
    for the method of iterated residues 
    to be applicable. 
Write $\omega = G(z)\,dz_1 \wedge \cdots \wedge dz_r$. 
Fix arbitrary $x_2,\ldots,x_r \in \RR$. 
At the first step, we require that 
\begin{equation} 
    \int_\RR G(z_1,x_2,\ldots,x_r) \, dz_1
    =
    \lim_{R \to \infty}
    \int_{C_R} G(z_1,x_2,\ldots,x_r) \, dz_1
\end{equation} 
where $C_R \subset \overline{\mathbb H}$ is 
    the positively oriented centered semi-circle of radius $R$ 
    resting on the real axis. 
By Cauchy's residue formula, 
\begin{equation} 
    \left(\int_\RR G(z_1,x_2,\ldots,x_r) \, dz_1\right)
    dz_2 \wedge \cdots \wedge dz_r
    =(2\pi i)\sum_{H}\res_z[\omega, H](x_2,\ldots,x_r)
\end{equation} 
where the sum extends over 
those polar hyperplanes $H$ 
for which $z_1^\ast$ is not only soluble 
but is in the upper half-plane. 
At the second step we write 
    $\res_z[\omega,H]=G_{H}(z_{(2)}) 
    \,dz_2 \wedge \cdots \wedge dz_r$ 
    and require for any such hyperplane $H$ that 
\begin{equation} 
    \int_\RR G_{H}(z_2,x_3,\ldots,x_r) \, dz_2
    =
    \lim_{R \to \infty}
    \int_{C_R} 
    G_{H}(z_2,x_3,\ldots,x_r) \, dz_2. 
\end{equation} 
By Cauchy's residue formula, 
\begin{equation} 
    \left(\int_\RR\int_\RR G(z_1,z_2,x_3,\ldots,x_r) \, dz_1 dz_2\right)
    dz_3 \wedge \cdots \wedge dz_r
    =(2\pi i)^2
    \sum_{H \supset H'}\itres_z[\omega, H \supset H \cap H'](x_3,\ldots,x_r)
\end{equation} 
where the sum is over those $2$-step polar flags 
$H \supset H \cap H'$ 
which are soluble and for which 
$$z_1^\ast(x_2,\ldots,x_r),
z_2^\ast(x_3,\ldots,x_r) \in \mathbb H.$$ 

\begin{definition} 
Assume that $h$ has an analytic continuation to 
    a neighborhood of $\Pi$ in $V_\CC$ 
    and the integral $\int_V \omega$ is absolutely convergent. 
For any polar flag $\gamma : H_1 \supset \cdots \supset H_1 \cap \cdots \cap H_{k-1}$ of $\omega$ 
    with $2\leq k \leq r$ 
    let $z_{(k)}=(z_{k},\ldots,z_r)$ and write 
    $\res[\omega,\gamma]=G_\gamma(z_{(k)}) \, dz_k \wedge \cdots \wedge dz_r$. 
The 
    \defn{iterated residue expansion of 
    the integral $\int_V\omega$ 
    along the polyhedron $\Pi$
    converges} 
    if for 
    every $x_2,\ldots,x_r \in \RR$ 
    and every polar flag 
    $\gamma: H_1 \supset \cdots 
    \supset H_1 \cap \cdots \cap H_{k-1}$ 
    arising in the iterated residue expansion 
    with $2\leq k\leq r$ 
    we have that 
\begin{equation} 
    \int_\RR 
    G_\gamma(z_{k},x_{k+1},\ldots,x_r) 
    \, dz_{k}
    =
    \lim_{R \to \infty}
    \int_{C_R} 
    G_\gamma(z_{k},x_{k+1},\ldots,x_r) 
        \, dz_{k}.
\end{equation} 
\end{definition} 

When this holds, 
it follows directly from Cauchy's residue formula that 
\begin{equation} 
    \frac{1}{(2\pi i)^{r}}
    \int_{V} \omega 
    = 
    \sum_{\gamma\in Z_\ast}
        \itres_z[\omega,\gamma] 
\end{equation}   
where $Z_\ast \subset Z$ 
is the distinguished subset of flags $\gamma$ 
which actually arise in 
the iterated residue expansion along $\Pi$. 
Since such flags must be soluble, we have that 
\begin{equation} 
    \sum_{\gamma\in Z_\ast}
        \itres_z[\omega,\gamma] 
    = 
    \sum_{\gamma\in Z_\ast}
        \res[\omega,\gamma] .
\end{equation}   

The problem is to determine $Z_\ast$ in terms of $\Pi$. 
Naively one might expect that 
    $Z_\ast=\{\gamma \in Z : \gamma(r) \in \Pi\}$ 
    but there are forms $\omega$ and polyhedra $\Pi$ 
    such that $\int_V \omega$ 
    has a convergent iterated residue expansion 
    along $\Pi$ 
    and $\int_V \omega \neq 0$ 
    and $\{\gamma \in Z: \gamma(r) \in \Pi\} = \varnothing$ 
    (\S\ref{sec:Example}).  
This occurs because the iterated residue expansion 
    along $\Pi$ may pick up poles 
    \emph{outside} of $\Pi$ 
    (Figure~\ref{fig:ExampleLeavingPolyhedron}). 
The compatibility condition on minors 
    prevents this from happening, 
    and is the minimal condition 
    which does so by Proposition~\ref{prop:converse}. 




\section{Stability}\label{sec:lemma} 

Here we study the behavior of residues under 
    a variation in the hyperplane parameters $s$. 
We write $\omega_s(z)$ for the form \eqref{eqn:omegaForm} 
with respect to some presentation $g_k = f_k-is_k$ 
(Lemma~\ref{lemma:elementary}) 
regarding the $s_1,\ldots,s_R$ as parameters. 
It turns out that some flags 
    do not contribute to the iterated residue expansion 
    for a non-empty open subset of parameters $s$ 
    whereas others contribute irrespective of $s$. 
We prove the latter flags are precisely 
    those cut out by a $\Pi$-stable collection of hyperplanes. 

\begin{lemma}\label{lemma}
Let $H_1,\ldots,H_k$ be 
    linearly independent polar hyperplanes. 
The flag 
    $\gamma(s): H_1 \supset H_1 \cap H_2 
    \supset \cdots \supset H_1 \cap \cdots \cap H_k$ 
    arises in the iterated residue expansion of 
    $\omega_s(z)$ 
    for all parameters $s$ with 
    $\mathrm{Re}\,s_1,\ldots,\mathrm{Re}\,s_R >0$ 
    if and only if $H=(H_1,\ldots,H_k)$ 
    is $\Pi$-stable. 
\end{lemma}

\begin{example} 
For a two-dimensional integral 
    with two polar hyperplanes $H_1$ and $H_2$, 
    the pole at $H_1 \cap H_2$ 
    is given by $m = z_1^\ast v_1 + z_2^\ast v_2$ where 
\begin{align} 
    z_1^\ast &= p_1^{-1}(is_1-z_2^\ast q_{12}),\\
    z_2^\ast &= ip_2^{-1}(p_1s_2-s_1r_{12}).
\end{align} 
The flag $H_1 \supset H_1 \cap H_2$ 
    arises in the iterated residue expansion 
    for all $s_1,s_2$ with positive real parts if and only if 
    $p_1,p_2>0$ and $r_{12} \leq 0$. 
For instance, consider the form 
\begin{equation} 
    \omega = \frac{dx \wedge dy}{(x^2+s_1^2)((x+y)^2+s_2^2)}
\end{equation} 
    with polar hyperplanes 
    $H_1 = \{x=is_1\}$, 
    $H_2 = \{x+y=is_2\}$, 
    $H_3 = \{-x=is_1\}$, 
    $H_4 = \{-x-y=is_2\}$. 
The Jacobian of $H=(H_1,H_2)$ with respect to 
    the standard polyhedron $\overline{\mathbb H}^2$ is 
\begin{equation} 
    J_H = \begin{bmatrix}
        1&0\\1&1
    \end{bmatrix},
    \qquad
    p_1=1,\,\,\,
    p_2=1,\,\,\,
    r_{12}=1,\,\,\,
    q_{12}=0\implies 
    \text{$J_H$ is unstable}.
\end{equation} 
The flag $\gamma_H(s) : H_1 \supset H_1 \cap H_2$ 
    contributes to 
    the iterated residue expansion of $\omega$ 
    along $\overline{\mathbb H}^2$ if and only if 
    $\mathrm{Re}\,s_2 \geq \mathrm{Re}\,s_1$. 
\end{example} 

\begin{proof} 
Set $\gamma = \gamma_H(s)$. 
The flag $\gamma$ 
    arises in the iterated residue expansion 
    of $\omega$ along $\Pi$ 
    if and only if 
    the values $z_1^\ast,\ldots,z_k^\ast$ 
    in the iterated residue expansion are soluble 
    and in $\mathbb H$ (cf.~\eqref{eqn:fk}) . 
We will show this occurs 
    for all $s_1,\ldots,s_R>0$ 
    if and only if $H=(H_1,\ldots,H_k)$ is $\Pi$-stable. 
Let $(w_k)_k \subset V$ be the dual basis to 
    $(f_k)_k\subset V^\vee$. 
Let $v=z_1^\ast v_1+\cdots+z_k^\ast v_k
    +x_{k+1}v_{k+1}+\cdots+x_rv_r$ 
and consider the following vector 
$$v^k \coloneq 
    (-1)^{k-1}
    v_1 \wedge \cdots \wedge v_{k-1}
    \wedge w_{k+1} \wedge \cdots \wedge w_r
    \in\wedge^{r-1} V.
$$ 
One finds that 
\begin{equation} 
    v \wedge v^k = 
    (z_k^\ast p_k +x_{k+1}q_{k,k+1}+\cdots+x_rq_{kr} )
    w_1 \wedge \cdots \wedge w_r\in \wedge^{r} V
\end{equation} 
and for any integer $j \leq k$ that 
\begin{equation} 
    w_j \wedge v^k = (-1)^{k-j}r_{jk}
    w_1 \wedge \cdots \wedge w_r \in \wedge^{r} V.
\end{equation} 
Thus 
$v-is_1w_1 - \cdots -is_kw_k \in 
\ker(v \mapsto v\wedge v^k)
=\mathrm{span}(v_1,\ldots,v_{k-1},w_{k+1},\ldots,w_r)$
if and only if 
\begin{equation}\label{eqn:zkFormula} 
    z_k^\ast= p_k^{-1}(i(s_kp_{k-1}-s_{k-1}r_{k-1,k}+\cdots+
    (-1)^{k-1}s_1r_{1k})
    -x_{k+1}q_{k,k+1}-\cdots-x_rq_{kr}). 
\end{equation} 
Since the coefficients of $v_1,\ldots,v_{k-1}$ are 
    determined by earlier residues, 
    and changing $v$ by multiples of $w_{k+1},\ldots,w_r$ 
    does not affect the condition that 
    $v \in H_1 \cap \cdots \cap H_k$, 
    the element $v-is_1w_1 - \cdots - is_kw_k$ 
    is in $\ker(v \mapsto v\wedge v^k)$ 
    if and only if $v \in H_1 \cap \cdots \cap H_k$. 
The rest follows by induction on $k$. 
\end{proof} 


The meaning of stability and compatibility 
    can now be summarized. 
Let $\CC^R_+ \subset \CC^R$ 
    denote the half-space with positive real parts. 
Again we write $\omega_s(z)$ for the form \eqref{eqn:omegaForm} 
with respect to some presentation $g_k = f_k-is_k$ 
(Lemma~\ref{lemma:elementary}) 
regarding the $s_1,\ldots,s_R$ as parameters in $\CC_+$. 

\begin{corollary}
Let $\gamma$ be a complete affine flag 
    whose linearization $z\gamma_0$ 
    is in the open Bruhat cell 
    $B^T \, \backslash \,B^T B \subset \mathscr F$. 
Let $z_k^\ast=z_k^\ast(x,s)$ 
    be given\footnote{Note $z_k^\ast$ only depends 
    on $\gamma$ and not 
    the particular collection of hyperplanes 
    used to define $\gamma$. 
    } 
    by \eqref{eqn:zkFormula}. 
Fix some $s_0 \in \CC^R_+$. 
Let $\Pi=z_\CC^{-1}(\overline{\mathbb H}^r)$ 
    be a polyhedron. 
    \begin{enumerate}
        \item The flag $\gamma$ arises in the iterated residue expansion of $\omega=\omega_{s_0}$ along $\Pi$ if and only if 
            $z_1^\ast(x,s_0),\ldots,z_r^\ast(x,s_0)$ 
            are in $\mathbb H$ 
            for all $x \in \RR^{r-1}$ 
            (resp. for some $x \in \RR^{r-1}$). 
        \item The flag $\gamma$ arises in the iterated residue expansion of $\omega_s$ along $\Pi$ for all $s \in \CC^R_+$ if and only if 
            $z_1^\ast(x,s),\ldots,z_r^\ast(x,s)$ are in 
            $\mathbb H$ 
            for all $s \in \CC^R_+$ 
            and $x \in \RR^{r-1}$ 
            (resp. some $x \in \RR^{r-1}$) 
            if and only if $\gamma$ is cut out by 
            a $\Pi$-stable collection of polar hyperplanes. 
        \item Assume $\gamma$ is $\Pi$-compatible. 
            The flag $\gamma$ arises in 
            the iterated residue expansion of 
            $\omega=\omega_{s_0}$ along $\Pi$ 
            if and only if 
            $z_1^\ast(x,s_0),\ldots,z_r^\ast(x,s_0)$ 
            are in $\mathbb H$ 
            for all $x \in \overline{\mathbb H}^{r-1}$. 
        \item Assume $\gamma$ is $\Pi$-compatible. 
            The flag $\gamma$ arises in 
            the iterated residue expansion of $\omega_{s}$ 
            along $\Pi$
            for all $s \in \CC^R_+$ 
            if and only if 
            $z_1^\ast(x,s),\ldots,z_r^\ast(x,s)$ 
            are in $\mathbb H$ 
            for all $s \in \CC^R_+$ and 
            $x \in \overline{\mathbb H}^{r-1}$ 
            if and only if $\gamma$ is cut out by 
            a $\Pi$-stable collection of polar hyperplanes. 
    \end{enumerate}
\end{corollary}

%


\section{Proof of the residue formula} 

\begin{theorem}
    Assume the iterated residue expansion of 
    $\int_{V}\omega$ along $\Pi$ 
    converges. 
If every polar flag of $\omega$ is $\Pi$-compatible, then  
\begin{equation}\label{eqn:ResidueFormula} 
    \frac{1}{(2\pi i)^{r}}
    \int_{V} \frac{h(z)\,dz}{g_1(z)\cdots g_R(z)}
    = 
    \sum_{\gamma \in Z_\Pi}
        \res[\omega,\gamma] 
\end{equation}   
where $Z_\Pi \subset Z$ 
    is the subset of flags cut out by 
    a $\Pi$-stable collection of polar hyperplanes. 
\end{theorem}

\begin{proof} 
Fix arbitrary $x_2,\ldots,x_r \in \RR$. 
For each linearly independent ordered collection 
    $H=(H_1,\ldots,H_r)$ of polar hyperplanes 
    with associated flag $\gamma=\gamma_H$ 
    let $z_k^\ast=z_k^\ast(x,s,\gamma)$ 
    be given by \eqref{eqn:zkFormula}. 
There is a dense Zariski-open subset $U \subset \CC_+^R$ 
    such that for all $s \in U$ 
    each polar flag of $\omega = \omega_s(z)$ is cut out by 
    a \emph{unique} ordered collection of 
    polar hyperplanes.\footnote{Equivalently, each terminal point of a polar flag is simple in the sense of 
    \cite[\S II.5.2]{zbMATH00107779}.} 
A flag $\gamma$ arises in the iterated residue expansion 
    if and only if $z_k^\ast(x,s,\gamma) \in \mathbb H$ 
    for all $k$. 
Thus for all $s \in U$ we have 
\begin{equation} 
    \frac{1}{(2\pi i)^{r}}
    \int_{V} \omega 
    = 
    \sum_H
        \itres_z[\omega,\gamma_H] \cdot 1_{\gamma_H}(s)
\end{equation} 
where the sum is over 
    all linearly independent $r$-sets $H$ of polar hyperplanes 
    and $1_{\gamma}$ is the characteristic function of 
    the open subset of parameters 
\begin{equation} 
    S_{\gamma}=\left\{s\in \CC^R_+ : 
    \mathrm{Im}(z_k^\ast(x,s,\gamma)) > 0 
    \text{ for all $1\leq k\leq r$}\right\}.
\end{equation} 
(Were $s \not \in U$ one would instead sum over flags.)
We restrict the sum to soluble flags 
    as insoluble flags do not contribute to 
    the iterated residue expansion of $\int_V \omega$. 
For such flags the truncated residues are classical residues: 
\begin{equation} 
    \frac{1}{(2\pi i)^{r}}
    \int_{V} \omega 
    =
    \sum_{z \gamma_{H,0}\in B^T \,\backslash \,B^T B}
        \itres_z[\omega,\gamma_H] \cdot 1_{\gamma_H}(s)
    = 
    \sum_{z \gamma_{H,0}\in B^T \,\backslash \,B^T B}
        \res[\omega,\gamma_H] \cdot 1_{\gamma_H}(s).
\end{equation} 
Write $H>_\Pi 0$ to mean $H$ is $\Pi$-stable. 
If $H >_\Pi 0$ then $S_{\gamma_H} = \CC^R_+$ 
    by Lemma~\ref{lemma} so for all $s \in U$ we have 
\begin{equation} 
    \frac{1}{(2\pi i)^{r}}
    \int_{V} \omega 
    =
    \sum_{H>_\Pi 0}
        \res[\omega,\gamma_H] 
    +
    \sum_{H \ngtr_\Pi 0}
        \res[\omega,\gamma_H] 
        \cdot 1_{\gamma_H}(s)
\end{equation}   
Since $\int_V \omega$ is absolutely convergent by assumption, 
    it varies analytically with the parameters $s\in \CC^R_+$. 
Meanwhile if $H>_\Pi 0$ then 
    the $z_1^\ast(x,s,\gamma_H),\ldots,z_r^\ast(x,s,\gamma_H)$ 
    are soluble for all $s \in U$ and so 
    $\res[\omega,\gamma_H]$ is 
    a meromorphic function of $s\in \CC^R_+$. 
Therefore 
\begin{equation}\label{eqn:Rvanishing} 
    R\coloneqq 
    \sum_{H \ngtr_\Pi 0}
        \res[\omega,\gamma_H] 
        \cdot 1_{\gamma_H}(s)
    =
    \frac{1}{(2\pi i)^{r}}
    \int_{V} \omega 
    - 
    \sum_{H>_\Pi 0}
        \res[\omega,\gamma_H] 
\end{equation} 
is a meromorphic function of $s \in U$. 
We claim the only way for 
    the piecewise function $R$ to be meromorphic 
    is if it is identically zero. 
Suppose it is not identically zero, 
in which case 
$\{S_{\gamma_H} : S_{\gamma_H} \neq \varnothing, H\not >_\Pi 0 \}$ 
is an open covering of $\CC^R_+$. 
Choose some $H$ for which 
    $S_{\gamma} \neq \varnothing$ and $H\not >_\Pi 0$ 
    where $\gamma = \gamma_H$. 
Lemma~\ref{lemma} implies $S_\gamma$ 
is a proper subset of $\CC_+^R$ 
so the complement $\CC^R_+ - S_\gamma$ 
has nonempty interior, 
and thus the topological boundary $\partial S_\gamma$ is nonempty. 

\begin{figure}[h]
\begin{tikzpicture}[scale=.7] 

\def\x{5}
\def\y{4}
\def\sa{2}
\def\sb{1}
\def\sd{1}

\draw (\x+2/3,\y-1) node {$S_\gamma$};

\draw (\x+4/5,\y-3) node {$\partial S_\gamma$};


\def\ld{1.5}
\def\op{.6}

    \draw[dashed,pattern={Lines[angle=-10,distance=\ld]},opacity=\op] 
    (0,0) -- (0,\y) 
        -- (\y,\y) -- (0,0);

    \draw[white, draw = white, thick] 
    (0,\y) -- (\y,\y); 

    \draw[dashed, pattern={Lines[angle=-55,distance=\ld]},opacity=\op] 
        (0,0) -- (3,\y) 
        -- (\x,\y) -- (\x,1) -- (0,0);

    \draw[white, draw = white, thick] 
    (3,\y) -- (\x,\y) -- (\x,1); 

    \draw[dashed, pattern={Lines[angle=70,distance=\ld]},opacity=\op] 
    (0,0) -- (\x,2) 
        -- (\x,-\y) -- (2*\y/3,-\y) -- (0,0);

    \draw[white, draw = white, thick] 
    (\x,2) -- (\x,-\y); 

    \draw[dashed, pattern={Lines[angle=35,distance=\ld]},opacity=\op] 
    (0,0) -- (0,-\y) 
        -- (\x,-4) -- (0,0);

    \draw[white, draw = white, thick] 
    (\x,-\y) -- (0,-\y); 

\draw[-] (0,-\y) -- (0,\y) node[anchor=east] {};

\draw (2*\x/3,2/3) node[circ,fill=white,inner sep=1pt]{}; 
\draw (1*\x/3,1/3) node[circ,fill=white,inner sep=1pt]{}; 

    \end{tikzpicture} 
    \caption{A real picture of $\CC_+^R$ covered by four unstable cones. Two points of $\partial S_\gamma - U$ 
    are marked on the boundary $\partial S_\gamma$ 
    of the depicted cone $S_\gamma$.}
\end{figure}

\noindent
The boundary $\partial S_\gamma$ 
    has real codimension one in $\CC^R_+$ 
    since $S_\gamma$ is a real open cone in $\CC^R_+$, 
so it cannot be covered by the complement of $U$ 
or the divisor where $R$ is singular since 
    these have real codimension two. 
Thus $R$ is regular at some boundary point 
    $z \in \partial S_\gamma \cap U$, 
    but this must also be a point of discontinuity for $R$. 
This contradiction implies $R = 0$. 
We conclude 
\begin{equation} 
    \frac{1}{(2\pi i)^{r}}
    \int_{V} \omega 
    = 
    \sum_{H>_\Pi 0}
        \res[\omega,\gamma_H] 
\end{equation} 
for all $s \in U$. 
Both sides admit meromorphic continuation to $\CC_+^R$. 
A meromorphic expression for the right-hand side 
valid for any $s\in\CC_+^R$ is 
\begin{equation} 
    \sum_{\gamma \in Z_\Pi}
        \res[\omega,\gamma]. 
\end{equation} 
This proves the formula. 
\end{proof} 

\subsection{A converse} 

The hypothesis of $\Pi$-compatibility cannot be dropped. 
Let $\Omega$ denote the set of all $r$-forms of the form 
    \eqref{eqn:presentationOfOmega} 
    for some $s\in\CC_+^R$. 
Our theorem is the `if' direction 
    of the next proposition 
    and we omit the easy proof of 
    the `only if' direction. 

\begin{proposition}\label{prop:converse}
    The residue formula \eqref{eqn:ResidueFormula} holds for all 
    $\omega \in \Omega$ for which $\int_V \omega$ has 
    a convergent iterated residue expansion along $\Pi$ 
    if and only if every flag formed from 
    the hyperplanes $H_1=\{g_1 = 0\},\ldots,H_R=\{g_R=0\}$ 
    is $\Pi$-compatible. 
\end{proposition}


\subsection{Sufficient conditions for convergence}\label{sec:Convergence} 

Since our convergence condition is formulated 
    in terms of one-dimensional integrals, 
    the classical Jordan lemma may be directly applied to give 
    sufficient conditions for convergence. 
These are easier to check 
    than \cite[Definition~2]{zbMATH00739368} 
    which involves auxiliary forms. 

\begin{proposition}
    Suppose that every polar flag of 
    $\omega = h(z)\,dz/(g_1\cdots g_R)$ is $\Pi$-compatible. 
\begin{enumerate}
\item If $R > r$ and $h(z)$ is 
        a bounded holomorphic function on $\Pi$, 
    then the iterated residue expansion of $\int_V \omega$ 
        along $\Pi$ converges. 
\item If $\psi\in V^\vee$ is a real linear form 
    satisfying $\psi(\theta)>0$ for each $\theta \in \Theta(1)$, 
    $h(z)$ is a holomorphic function on $\Pi$ 
    satisfying $h(z) = o(|z|^R)$, 
    and $\int_V \omega$ is absolutely convergent, 
    then the iterated residue expansion of $\int_V \omega$ 
        along $\Pi$ converges. 
    \end{enumerate}
\end{proposition}

One might naively expect (1) to hold 
    even without assuming compatibility, 
    however Example~\ref{sec:Example} 
    (when $n_1 > n_2$) 
    shows that compatibility cannot be dropped. 


\subsection{Grothendieck residues}\label{sec:ClassicalResidue} 

Although it is not necessary to apply our residue formula, 
    here we explain how to express it 
    using Grothendieck residues. 
We recall the definition. 
Consider a meromorphic $r$-form $\omega$ 
    on an open set $U \subset V_\CC$. 
By a \defn{system of divisors of $\omega$ at a point $m$} 
    we mean a collection $D = (D_1,\ldots,D_r)$ 
    of divisors in $U$ such that $\omega$ 
    is regular away from the union 
    $D_1 \cup \cdots \cup D_r$ 
    and $D_1 \cap \cdots \cap D_r$ 
    is a discrete set containing $m$. 
Suppose that $D_k = \{g_k = 0\}$ and 
define the topological $r$-cycle 
\begin{equation}\label{eqn:rCycle} 
    \Gamma_{g} = \{z \in V_\CC:|g_j(z)| = \varepsilon_j 
    \text{ for all $j=1,\ldots,r$}\}
\end{equation} 
where the positive constants $\varepsilon_j$ 
are sufficiently small that 
$\Gamma_g \subset U$; 
this $r$-cycle is equipped with 
its unique orientation for which 
$d(\mathrm{arg}\,g_1)\wedge \cdots \wedge d(\mathrm{arg}\,g_r) \geq 0$. 

\begin{definition}\label{defn:Residue}
    The \defn{residue of the form $\omega$ 
    with respect to the system of divisors 
    $D$ at $m$} is 
\begin{equation} 
    \mathrm{res}_D[\omega,m]=
    (2\pi i)^{-r}
    \int_{\Gamma_g}
    \omega.
\end{equation} 
\end{definition}

\begin{remark}
$\res_D[\omega,m]$ is independent of the $\varepsilon_j$ 
    by Stokes's theorem, 
    cf.~\cite[\S II.5.1]{zbMATH00107779}. 
\end{remark}

\begin{remark}
In the algebraic geometry literature 
    this is also denoted 
    $$\mathrm{res}\left[\genfrac{}{}{0pt}{}{h\,dz}{g_1,\ldots,g_r} \right]$$ 
    where $\omega=h\,dz/(g_1\cdots g_r)$. 
    It admits a purely algebraic definition 
    \cite[\S III.9]{zbMATH03336606}. 
\end{remark}

One is tempted to regard $\mathrm{res}_D[\omega,m]$ 
    as determined by $\omega$ and $m$, 
    but it crucially depends on 
    the system of divisors. 
If the divisors are irreducible, 
    then this ambiguity amounts to a sign. 
However, if any of the divisors is \emph{reducible}, 
    then there are multiple ways to group 
    the irreducible singular divisors into $r$ 
    divisors $D_1,\ldots,D_r$ 
    and {different groupings generally result in 
    independent residues}. 
See \S\ref{sec:Example2} for an integral 
    with three different groupings $g$ 
    leading to three non-homologous cycles $\Gamma_g$. 

Let $z = (z_1,\ldots,z_r)$ be a basis of $V^\vee_\CC$. 
For each polar flag $\gamma :
H_{1} \supset 
H_{1} \cap H_{2} \supset \cdots \supset
H_{1} \cap \cdots \cap H_{r}$ 
with defining equations 
    $H_k = \{g_k=0\}$, 
    let $\partial g_\gamma/\partial z$ be the Jacobian matrix of 
    $(g_1,\ldots,g_r) \colon V_\CC \to \CC^r$ 
    with respect to $z$. 

\begin{proposition}\label{prop:GrothtoIter}
Let $D = (D_1,\ldots,D_r)$ 
    be a system of divisors of $\omega$ at $m$. 
Let $F \subset Z$ denote the set of flags $\gamma$ 
which terminate at $m$ and arise from $D$ 
    in the sense that 
$\gamma :
H_{1} \supset 
H_{1} \cap H_{2} \supset \cdots \supset
H_{1} \cap \cdots \cap H_{r}$ 
and $H_k \subset D_k$ for all $1 \leq k \leq r$. 
If $\partial g_\gamma/\partial z \in B^T B$ 
    for every $\gamma\in F$, then 
\begin{equation} 
    \res_D[\omega,m]
    =\sum_{\gamma\in F} \itres_z[\omega,\gamma].
\end{equation} 
    Set $H = (H_1,\ldots,H_r)$ and let $\gamma = \gamma_H$. 
Then 
\begin{equation} 
    \itres_z[\omega,\gamma] = 
    \res_H[\omega,m]1_{B^TB}(\partial g_\gamma/\partial z)
\end{equation} 
where $1_{B^TB}$ 
    is the characteristic function 
    of $B^T B$. 
\end{proposition}

\begin{proof} 
Assume $\partial g_\gamma/\partial z\in B^T B$ 
    for every $\gamma\in F$. 
Fix $x_2,\ldots,x_r \in \CC$. 
By Proposition~\ref{prop:openBruhat} 
    each of the values $z_1^\ast,\ldots,z_k^\ast$ 
    is soluble for every $\gamma \in F$. 
Since $z_1^\ast$ is soluble for 
    the first residue of each flag $\gamma \in F$, 
    the slice of the cycle $\Gamma_g$ 
    with coordinates $x_2,\ldots,x_r$ 
    is a union of simple loops 
    in the $z_1$-complex plane 
    around these $z_1^\ast$ values. 
By Cauchy's residue theorem 
    $(2\pi i)^{-r}
    \int_{\Gamma_g}
    \omega = (2\pi i)^{-(r-1)}
    \sum_{H \subset D_1}
    \itres_{z_1}[\omega,H](x_2,\ldots,x_r)$. 
The rest of the first claim follows by induction. 
For the second claim, 
    if $\partial g_\gamma/\partial z\not \in B^T B$ 
    then we are done by Proposition~\ref{prop:openBruhat}, 
and if $\partial g_\gamma/\partial z \in B^T B$ 
then the formula follows from the first claim. 
\end{proof} 

Now we may explain how to obtain a system of divisors 
    satisfying \cite[Definition~1]{zbMATH00739368} 
    from a meromorphic form whose polar flags are $\Pi$-compatible. 

\begin{corollary}\label{cor:Comparison} 
    Let $\mathcal H$ be the set of all collections 
    $H = (H_1,\ldots,H_r)$ 
    of polar hyperplanes of $\omega$ 
    giving rise to any one of the flags $\gamma \in Z_\Pi$. 
Let $D_k = \cup_{H \in \mathcal H}H_k$ and set 
    $D = (D_1,\ldots,D_r)$. 
Then $D_1 \cap \cdots \cap D_r$ is discrete and 
\begin{equation} 
    \sum_{m \in D_1 \cap \cdots \cap D_r}\res_D[\omega,m]
    =\sum_{\gamma\in Z_\Pi} \itres_z[\omega,\gamma].
\end{equation} 
If every polar flag of $\omega$ is $\Pi$-compatible, then 
    $D$ satisfies \cite[Definition~1]{zbMATH00739368}. 
If additionally $\int_V \omega$ has 
    a convergent iterated residue expansion 
    along the polyhedron $\Pi$, then  
\begin{equation} 
    \frac{1}{(2\pi i)^{r}}
    \int_{V} \frac{h(z)\,dz}{g_1(z)\cdots g_R(z)}
    = 
    \sum_{m \in D_1 \cap \cdots \cap D_r}\res_D[\omega,m].
\end{equation}   
\end{corollary} 



\section{Two examples} 

We compute two two-dimensional integrals with the residue formula. 
The relevant minors for a $2\times 2$ matrix 
$\begin{bmatrix}a&b\\c&d\end{bmatrix}$ 
    are $p_1=a$, $p_2 = ad-bc$, $r_{12}=c$, and $q_{12} = b$. 
The standard basis of $\RR^2$ is denoted by $e_1,e_2$. 

\subsection{Example 1}\label{sec:Example} 

\FloatBarrier


Let $n_1,n_2 \geq 1$ be real parameters and let $s_1,s_2,s_3>0$ 
be complex numbers with positive real parts. 
Consider the integral 
\begin{equation} 
    \int_{\RR^2} 
    \frac{n_1^{i x-s_1}n_2^{iy-s_2} \,dx\wedge dy}
                {(-x-is_1)(-y-is_2)(x+y-is_3)}.
\end{equation} 
There are six flags in $Z$ formed from 
    the three polar hyperplanes $H_1,H_2,H_3$, 
    and three terminal points 
    $z_{13},z_{23},z_{12}$ 
    where any two hyperplanes meet. 
Let $\gamma_{ij} : H_i \supset H_i \cap H_j$ 
    and consider the three polyhedra 
\begin{align} 
    \Pi_A &=\RR^2+i\Theta_A 
    = \RR^2 + i\RR_{\geq 0}\langle e_1,e_2\rangle 
    = \overline{\mathbb H}^2,\\
    \Pi_B&=\RR^2+i\Theta_B=
    \RR^2+i\RR_{\geq 0}\langle-e_1+e_2, e_2\rangle,\\
    \Pi_C&=\RR^2+i\Theta_C=
    \RR^2+i\RR_{\geq 0}\langle e_1-e_2,e_1 \rangle. 
\end{align} 
\begin{figure} 
\begin{tikzpicture}[scale=.7] 

\def\x{5}
\def\y{4}
\def\sa{2}
\def\sb{1}
\def\sd{1}




    \fill[fill=gray!10,fill opacity=.2,pattern=dots] 
        (0,0) -- (0,\y) 
        -- (\x,\y) -- (\x,0);
    \draw (\x/2,\y/2) node {$\Theta_A$};
    \fill[fill=gray!22,fill opacity=.5] (0,0) -- (0,\y) 
        -- (-\y,\y) -- (0,0);
    \draw (-\x/3,2*\y/3) node [anchor=west] {$\Theta_B$};
    \fill[fill=gray!22,fill opacity=.5] (0,0) -- (\x,0) 
        -- (\x,-\y) -- (\y,-\y) -- (0,0);
    \draw (2*\y/3,-\x/4-.1) node [anchor=west] {$\Theta_C$};

\draw[->] (-\x,0) -- (\x,0) node[anchor=north] {};

\draw[->] (0,-\y) -- (0,\y) node[anchor=east] {};

\draw[thick,red] (-\sa,-\y) -- (-\sa,\y); 
    \draw (-\sa,\y+0.4) node {$\mathrm{Im}\,H_1$}; 
\draw[thick,red] (-\x,-\sb) -- (\x,-\sb); 
    \draw (\x,-\sb+0.4) node {$\mathrm{Im}\,H_2$}; 
\draw[thick,red] (\x,-\x+\sd) -- (\x-8,-\x+\sd+8); 
    \draw (\x-1,-\x+\sd) node {$\mathrm{Im}\,H_3$}; 

\draw (-\sa,-\sb) node[circle,red,fill,inner sep=1pt]{}; 
\draw (-\sa,\sd+\sa) node[circle,red,fill,inner sep=1pt]{}; 
\draw (\sd+\sb,-\sb) node[circle,red,fill,inner sep=1pt]{}; 


    \end{tikzpicture} 
    \caption{The configuration of polar hyperplanes and three polyhedra in imaginary space.}
\end{figure} 

\def\twd{1.5em}

\begin{center}
\begin{tabular}{ | m{6.7em} | m{\twd}| m{\twd} | m{\twd} | m{\twd}| m{\twd}| m{\twd}| } 
  \hline
    & $\gamma_{12}$ & $\gamma_{13}$ & $\gamma_{21}$ 
    & $\gamma_{23}$ & $\gamma_{31}$ & $\gamma_{32}$ \\ 
  \hline
  $\Pi_A$-stable & No & No & No & No & Yes & No\\ 
  $\Pi_A$-compatible & Yes & Yes & Yes & Yes & No & Yes\\ 
  \hline
  $\Pi_B$-stable & No & Yes & No & No & No & No\\ 
  $\Pi_B$-compatible & Yes & Yes & Yes & Yes & Yes & Yes\\ 
  \hline
  $\Pi_C$-stable & No & No & No & Yes & No & No\\ 
  $\Pi_C$-compatible & Yes & Yes & Yes & Yes & Yes & Yes\\ 
  \hline
\end{tabular}
\end{center}
\vspace{1em}

\subsubsection{An incompatible polyhedron} 

We will see that 
    the residue formula 
    with respect to the naive choice $\Pi_A$ does not hold since 
    $\{\gamma \in Z: \gamma(2) \in \Pi_A\} = \varnothing$ 
    yet $\int_{\RR^2} \omega \neq 0$. 
The hypotheses of the residue formula are not met 
since the flag $\gamma_{31}$ with defining matrix 
\begin{equation} 
    \frac{\partial f}{\partial z_A}=
    \frac{\partial f}{\partial x}=
    \begin{bmatrix}
        f_3(e_1)&f_3(e_2)\\
        f_1(e_1)&f_1(e_2)
    \end{bmatrix}
    =
    \begin{bmatrix}
        \phantom{-}1&1\\
        -1&0\\
    \end{bmatrix},
    \qquad
    p_1=1,\,\,\,
    p_2=1,\,\,\,
    q_{12}=1,\,\,\,
    r_{12}=-1
\end{equation} 
is not $\Pi_A$-compatible. 
Nonetheless 
we claim that $\int_{\RR^2} \omega$ has 
a convergent iterated residue expansion 
along $\Pi_A$ if $n_2 \geq n_1$. 
The first step of the iterate residue expansion 
    when $x$ is deformed to the upper half-plane 
    meets $H_3$ when $x=-y+is_3$ 
    (see the first path segment 
    in Figure~\ref{fig:ExampleLeavingPolyhedron}). 
There is no issue with convergence since 
    $n_1^{i x-s_1}n_2^{iy-s_2}$ is bounded in $\Pi_A$, 
    and the integral becomes 
\begin{equation} 
    (2\pi i)\int_{\RR} 
    \frac{n_1^{-iy-s_3-s_1}n_2^{iy-s_2} \,dy}
                {(y-is_3-is_1)(-y-is_2)}.
\end{equation} 
However the residual form is restricted to $H_3$ 
    where the magnitude of 
    the numerator is proportional to 
\begin{equation} 
    |n_1^{-iy}n_2^{iy}|
    =
    (n_1/n_2)^{\mathrm{Im}(y)}.
\end{equation} 
If $n_1 > n_2$, 
    then this blows up as $y$ is deformed to the upper half-plane, 
    and something noteworthy has occurred: 
    although the integrand initially 
    has good decay in the polyhedron,  
    at the second step of the iterated residue expansion 
    the residual form on $H_3$ 
    leaves the polyhedron $\Pi_A$ when 
    the second coordinate is deformed across $\mathbb H$ 
    (see the second path segment to $z_{13}$ in 
    Figure~\ref{fig:ExampleLeavingPolyhedron}).\footnote{
        The same situation arises in Langlands's calculation of the residual part of the discrete automorphic spectrum for the exceptional group of type $G_2$ via iterated residues of Eisenstein series 
    (which, as is well-known, have hyperplane singularities): 
``\emph{A number of difficulties can enter at the later stages which do not appear at first. The functions $M(s,\lambda)$ remain bounded as $\mathrm{Im}(\lambda) \to \infty$ in the region defined by (2) so that the application of the residue theorem is clearly justified. However, at least in the general case when the functions $M(s,\lambda)$ are not explicitly known, it was necessary to deform the contour into regions in which, so far as I could see, the behaviour of the relevant functions as $\mathrm{Im}(\lambda) \to \infty$ was no longer easy to understand. Some substitute for estimates was necessary. It is provided by unpleasant lemmas, such as Lemma 7.1, and the spectral theory of the operator $A$ introduced in \S6.}'' \cite[Appendix III, p.~189]{langlands}. 
    } 

If $n_2 \geq n_1$ then the second step converges, 
    and we see that compatibility 
    is necessary for the residue formula to hold: 
    the integral is equal to 
\begin{equation} 
        \frac{i(2\pi i)^2 }{s_1+s_2+s_3}n_2^{-(s_1+s_2+s_3)},
\end{equation} 
whereas the residue formula would predict the answer is $0$ 
since $\Pi_A$ contains no terminal points. 
Next we show how to evaluate the integral 
for any $n_1,n_2\geq 1$ with the residue formula by choosing 
    a compatible polyhedron. 

\begin{figure} 
\begin{tikzpicture}[scale=.7] 

\def\x{5}
\def\y{4}
\def\sa{2}
\def\sb{1}
\def\sd{1}




    \fill[fill=gray!10,fill opacity=.2,pattern=dots] 
        (0,0) -- (0,\y) 
        -- (\x,\y) -- (\x,0);
    \draw (\x/2,\y/2) node {$\Theta$};

\draw[->] (-\x,0) -- (\x,0) node[anchor=north] {};

\draw[->] (0,-\y) -- (0,\y) node[anchor=east] {};

\draw[thick,red] (-\sa,-\y) -- (-\sa,\y); 
\draw[thick,red] (\x,-\x+\sd) -- (\x-8,-\x+\sd+8); 

\draw[thick,->] (0,0) -- (\sd/2,0);
\draw[thick,->] (\sd/2,0) -- (\sd,0) -- (-\sa/2,\sd+\sa/2) ;
\draw[thick] (-\sa/2,\sd+\sa/2) -- (-\sa,\sd+\sa) ;

\draw (-\sa,\sd+\sa) node [anchor=north east] 
    {$z_{13}$}; 
\draw (-\sa,\sd+\sa) node[circle,red,fill,inner sep=1pt]{}; 


    \end{tikzpicture} 
    \caption{An example of incompatibility.}
    \label{fig:ExampleLeavingPolyhedron}
\end{figure} 


\subsubsection{A compatible polyhedron} 

There is no single polyhedron that works for all $n_1,n_2 \geq 1$, 
    but the following choices work: 
\begin{equation} 
    \Pi = 
    \begin{cases}
        \Pi_B=\RR^2+i\Theta_B=
            \RR^2+i\RR_{\geq 0}\langle-e_1+e_2, e_2\rangle
            &\text{if $n_2 \geq n_1$,}\\
        \Pi_C=\RR^2+i\Theta_C=
        \RR^2+i\RR_{\geq 0}\langle e_1-e_2,e_1 \rangle
        &\text{if $n_1 \geq n_2$}.
    \end{cases}
\end{equation} 
The only $\Pi_B$-stable collection of two hyperplanes 
    is $(H_1,H_3)$ for which 
\begin{equation} 
    J_{(H_1,H_3)}=
    \begin{bmatrix}
        \phantom{-}1&-1\\
        -1&\phantom{-}2\\
    \end{bmatrix},
    \qquad
    p_1=1,\,\,\,
    p_2=1,\,\,\,
    q_{12}=-1,\,\,\,
    r_{12}=-1.
\end{equation} 
Similarly for $n_1 \geq n_2$ 
    the only $\Pi$-stable collection of two hyperplanes 
    is $(H_2,H_3)$. 
In either case, 
    each of the six polar flags is compatible 
    with the specified polyhedra. 
The iterated residues for these flags are 
\begin{equation} 
    \res[\omega,\gamma_{13}]
    =
        \frac{in_2^{-(s_1+s_2+s_3)}}
                {s_1+s_2+s_3} 
    ,\qquad
    \res[\omega,\gamma_{23}]
    =
        \frac{in_1^{-(s_1+s_2+s_3)}}
                {s_1+s_2+s_3}.
\end{equation} 
Thus 
\begin{equation} 
    \left(\frac{1}{2\pi i}\right)^2
    \int_{\RR^2} 
    \frac{n_1^{i x-s_1}n_2^{iy-s_2} \,dx\wedge dy}
                {(-x-is_1)(-y-is_2)(x+y-is_3)}
    =
        \frac{i}{s_1+s_2+s_3}\max(n_1,n_2)^{-(s_1+s_2+s_3)}.
\end{equation} 
\begin{figure} 
    \centering
    \begin{minipage}{0.47\textwidth}
\begin{tikzpicture}[scale=.7] 

\def\x{5}
\def\y{4}
\def\sa{2}
\def\sb{1}
\def\sd{1}




    \fill[fill=gray!22,fill opacity=.5] (0,0) -- (0,\y) 
        -- (-\y,\y) -- (0,0);
    \draw (-\x/3,2*\y/3) node [anchor=west] {$\Theta_B$};

\draw[->] (-\x,0) -- (\x,0) node[anchor=north] {};

\draw[->] (0,-\y) -- (0,\y) node[anchor=east] {};

\draw[thick,red] (-\sa,-\y) -- (-\sa,\y); 
\draw[thick,red] (\x,-\x+\sd) -- (\x-8,-\x+\sd+8); 

\draw[thick,->] (0,0) -- (-\sa/2,\sa/2);
\draw[thick,->] (-\sa/2,\sa/2) -- (-\sa,\sa) -- 
    (-\sa,\sd/2+\sa) ;
\draw[thick] (-\sa,\sd/2+\sa) -- (-\sa,\sd+\sa) ;

\draw (-\sa,\sd+\sa) node [anchor=north east] 
    {$z_{13}$}; 
\draw (-\sa,\sd+\sa) node[circle,red,fill,inner sep=1pt]{}; 


    \end{tikzpicture} 
    \end{minipage}%
    \hspace{0.06\textwidth}%
    \begin{minipage}{0.47\textwidth}
\begin{tikzpicture}[scale=.7] 

\def\x{5}
\def\y{4}
\def\sa{2}
\def\sb{1}
\def\sd{1}




    \fill[fill=gray!22,fill opacity=.5] (0,0) -- (\x,0) 
        -- (\x,-\y) -- (\y,-\y) -- (0,0);
    \draw (2*\y/3,-\x/4-.1) node [anchor=west] {$\Theta_C$};

\draw[->] (-\x,0) -- (\x,0) node[anchor=north] {};

\draw[->] (0,-\y) -- (0,\y) node[anchor=east] {};

\draw[thick,red] (-\x,-\sb) -- (\x,-\sb); 
\draw[thick,red] (\x,-\x+\sd) -- (\x-8,-\x+\sd+8); 

\draw[thick,->] (0,0) -- (\sb/2,-\sb/2);
\draw[thick,->] (\sb/2,-\sb/2) -- (\sb,-\sb) -- 
    (\sd/2+\sb,-\sb) ;
\draw[thick] (\sd/2+\sb,-\sb) -- (\sd+\sb,-\sb) ;

\draw (\sd+\sb,-\sb-.1) node [anchor=north] 
    {$z_{23}$}; 
\draw (\sd+\sb,-\sb) node[circle,red,fill,inner sep=1pt]{}; 


    \end{tikzpicture} 
    \end{minipage}
    \caption{Staying in the polyhedron when $n_2 \geq n_1$ (left) 
    or $n_2 \leq n_1$ (right).}
    \label{fig:ExampleStayingInPolyhedron}
\end{figure} 



\subsection{Example 2}\label{sec:Example2} 

Consider the integral 
\begin{equation}
    \int_{\RR^2} 
    \frac{h(x,y)\,dx\wedge dy}{(x-i)(y-i)(x+y-2i)}.
\end{equation} 
The unique terminal point is not simple. 
In this example we will illustrate Proposition~\ref{prop:GrothtoIter}. 
\begin{figure}[h!] 
\begin{tikzpicture}[scale=.7] 

\def\x{5}
\def\y{4}
\def\sa{1}
\def\sb{1}
\def\sd{2}




    \fill[fill=gray!22,fill opacity=.5] (0,0) -- (\x,0) 
        -- (\x,\y) -- (-\y,\y) -- (0,0);
    \draw (\x/3,\y/2) node [anchor=west] {$\Theta$};

\draw[->] (-\x,0) -- (\x,0) node[anchor=north] {};

\draw[->] (0,-\y) -- (0,\y) node[anchor=east] {};

\draw[thick,red] (\sa,-\y) -- (\sa,\y) ; 
\draw[thick,red] (-\x,\sb) -- (\x,\sb); 
\draw[thick,red] (\x,-\x+\sd) -- (\x-7,-\x+\sd+7); 
    \draw (\sa,\y+0.4) node {$\mathrm{Im}(H_1)$}; 
    \draw (\x,\sb+0.4) node {$\mathrm{Im}(H_2)$}; 
    \draw (\x-1.2,-\x+\sd) node {$\mathrm{Im}(H_3)$}; 


\draw (\sa,\sa) node[circle,red,fill,inner sep=1pt]{}; 

    \end{tikzpicture} 
    \caption{The configuration of polar hyperplanes in imaginary space.}
    \label{fig:Example2Hyperplanes}
\end{figure} 

\subsubsection{Grothendieck residues} 

There are three groupings of 
    the polar hyperplanes into two divisors, 
    so there are three corresponding residues. 
We will see that 
\begin{alignat}{2} 
    \mathrm{res}
    \left[\genfrac{}{}{0pt}{}{h\,dx \wedge dy}{H_3H_1,H_2} \right]
    &=&&\partial_x h(i,i),\\
    \mathrm{res}
    \left[\genfrac{}{}{0pt}{}{h\,dx \wedge dy}{H_3H_2,H_1} \right]
    &=-&&\partial_y h(i,i),\\
    \mathrm{res}
    \left[\genfrac{}{}{0pt}{}{h\,dx \wedge dy}{H_1H_2,H_3} \right]
    &=&&\partial_y h(i,i).
\end{alignat} 
Thus the $2$-cycles $\Gamma$ 
    defined by \eqref{eqn:rCycle} 
    for these three groupings are non-homologous. 
For the first grouping 
    the $2$-cycle $\Gamma$ whose pairing with 
    the closed form $\omega$ gives the residue is 
    the topological $2$-torus 
\begin{equation} 
    \Gamma = \{(x,y) \in \CC^2 : 
    |(x+y-2i)(x-i)|=\varepsilon_1,|y-i|=\varepsilon_2|\}.
\end{equation} 
It is oriented so that we first integrate along 
    $x$ and then $y$, both in a positive sense. 
The integral over $x$ contributes two residues, 
and the remaining integral over $y$ is 
a positive loop around $y = i$. 
This obtains 
\begin{equation} 
    \mathrm{res}
    \left[\genfrac{}{}{0pt}{}{h\,dx \wedge dy}{H_3+H_1,H_2} \right]
    =
    (2\pi i)^{-1}\int
    \left(
    \frac{h(i,y)}{(y-i)^2}
    -
    \frac{h(-y+2i,y)}{(y-i)^2}
    \right)dy
    =
    \partial_x h(i,i).
\end{equation} 
The local residue for the grouping $(H_3H_2,H_1)$ 
    is similarly computed. 
For the third grouping $(H_1H_2, H_3)$, 
    the first integral over $x$ only contributes a single residue 
    at $x = i$, and the remaining integral over $y$ 
    is again a positive loop around $y = i$, resulting in 
\begin{equation} 
    \mathrm{res}
    \left[\genfrac{}{}{0pt}{}{h\,dx \wedge dy}{H_1H_2,H_3} \right]
    =
    (2\pi i)^{-1}\int
    \frac{h(i,y)\,dy}{(y-i)^2}
    =
    \partial_y h(i,i).
\end{equation} 


\subsubsection{Computing the integral using the residue formula} 

We use the polyhedron 
\begin{equation} 
    \Pi = \RR^2 + i\Theta,\qquad 
    \Theta=\RR_{\geq 0}\langle e_1,-e_1+e_2\rangle. 
\end{equation} 
We have the compatible and stable Jacobians 
\begin{equation} 
    J_{(H_1,H_2)}=
    \begin{bmatrix}
        1&-1\\
        0&\phantom{-}1\\
    \end{bmatrix},\quad
    J_{(H_1,H_3)}=
    \begin{bmatrix}
        1&-1\\
        1&\phantom{-}0\\
    \end{bmatrix},\quad
    J_{(H_3,H_2)}=
    \begin{bmatrix}
        1&0\\
        0&1\\
    \end{bmatrix}.
\end{equation} 
The three other ordered pairs of hyperplanes 
    are unstable so are compatible by definition. 
There are only two flags in $Z_\Pi$, 
\begin{equation} 
    \gamma_{12} = \gamma_{13} : 
    H_1 \supset \{(i,i)\},\qquad
    \gamma_{32} : 
    H_3 \supset \{(i,i)\} ,
\end{equation} 
with iterated residues 
(taking $z = \Theta(1)=\langle e_1,-e_1+e_2\rangle$) 
\begin{equation} 
    \itres_z[\omega,\gamma_{12}]=\partial_y h(i,i),\qquad
    \itres_z[\omega,\gamma_{32}]=\partial_x h(i,i)-\partial_y h(i,i).
\end{equation} 
The system of divisors at $(i,i)$ 
determined by these flags is $D =(H_3H_1,H_2)$ 
(cf.~Corollary~\ref{cor:Comparison}). 

Now let $h$ be any holomorphic function on $\Pi$ 
    which decays sufficiently quickly so that 
    the iterated residue expansion along $\Pi$ converges 
    (e.g.~$h=e^{2\pi i(x+2y)}$). 
Then 
\begin{align}
    (2\pi i)^{-2}
    \int_{\RR^2} 
    \frac{h(x,y)\,dx\wedge dy}{(x-i)(y-i)(x+y-2i)}
    =
    \itres_z[\omega,\gamma_{12}]
    +
    \itres_z[\omega,\gamma_{32}]
    =
    \partial_x h(i,i).
\end{align} 

%
%
%




\section*{Acknowledgements} 

The author is grateful to Peter Sarnak 
    for helpful discussions 
    and pointing out the relevance of \cite{langlands}. 
The author also thanks Roman Ulvert 
    and Lauren Williams for some helpful comments. 


\bibliography{residue}

\begin{thebibliography}{1}

\bibitem{zbMATH01353487}
V.~V. Batyrev and Y.~Tschinkel.
\newblock Manin's conjecture for toric varieties.
\newblock {\em J. Algebr. Geom.}, 7(1):15--53, 1998.

\bibitem{cubic}
S.~Bhattacharya and A.~O'Desky.
\newblock On monic abelian trace-one cubic polynomials, 2023.

\bibitem{zbMATH03336606}
R.~Hartshorne.
\newblock {\em Residues and duality. {Lecture} notes of a seminar on the work
  of {A}. {Grothendieck}, given at {Harvard} 1963/64.}, volume~20 of {\em Lect.
  Notes Math.}
\newblock Springer, Cham, 1966.

\bibitem{langlands}
R.~P. Langlands.
\newblock {\em On the functional equations satisfied by {Eisenstein} series},
  volume 544 of {\em Lect. Notes Math.}
\newblock Springer, Cham, 1976.

\bibitem{zbMATH00739368}
M.~Passare, A.~Tsikh, and O.~Zhdanov.
\newblock A multidimensional {Jordan} residue lemma with an application to
  {Mellin}--{Barnes} integrals.
\newblock In {\em Contributions to complex analysis and analytic geometry.
  Based on a colloquium dedicated to Pierre Dolbeault, Paris, France, June
  23-26, 1992}, pages 233--241. Braunschweig: Vieweg, 1994.

\bibitem{zbMATH00107779}
A.~K. Tsikh.
\newblock {\em Multidimensional residues and their applications. {Transl}. from
  the {Russian} by {E}. {J}. {F}. {Primrose}. {Transl}. edited by {S}.
  {Gelfand}}, volume 103 of {\em Transl. Math. Monogr.}
\newblock Providence, RI: American Mathematical Society, 1992.

\end{thebibliography}
\bibliographystyle{abbrv}

\end{document}